\documentclass[11pt,a4paper]{article}

\usepackage{amsfonts, amsmath, amssymb, amsthm,url,dsfont,multicol}

\usepackage{geometry}
\usepackage{graphicx,xcolor}
\usepackage[english] {babel}

\usepackage{tikz}
\usepackage{array}
\usepackage{cellspace}
  \newtheorem{theorem}{Theorem}[section]
   \newtheorem{lemma}[theorem]{Lemma}
   \newtheorem{proposition}[theorem]{Proposition}
   \newtheorem{corollary}[theorem]{Corollary}

\newcommand{\ben}{\begin{enumerate}}
\newcommand{\een}{\end{enumerate}}

\newcommand{\bt}{\begin{theorem}}
\newcommand{\et}{\end{theorem}}
\newcommand{\bl}{\begin{lemma}}
\newcommand{\el}{\end{lemma}}
\newcommand{\bc}{\begin{corollary}}
\newcommand{\ec}{\end{corollary}}
\newcommand{\bp}{\begin{proposition}}
\newcommand{\ep}{\end{proposition}}
\newcommand{\br}{\begin{remark}}
\newcommand{\er}{\end{remark}}

\newcommand{\bpf}{\begin{proof}}
\newcommand{\epf}{\end{proof}}

\begin{document}

\title{The asymptotic distribution of the remainder in a certain base-$\beta$ expansion}

\author{I.~W.~Herbst, J. M\o ller, A. M. Svane}

\date{}

\maketitle

\begin{abstract} 
Let $X=\sum_{k=1}^\infty X_k \beta^{-k}$ be the base-$\beta$ expansion of a continuous random variable $X$ on the unit interval where $\beta$ is the golden ratio. We study the asymptotic distribution and convergence rate of the scaled remainder $\sum_{k=n+1}^\infty X_k \beta^{n-k}$ when $n$ tends to infinity. 
\end{abstract}

\section{Introduction}
\label{s:intro}

Let $X$ be a random variable so that $0\le X<1$, and for $\beta>1$ consider its 
base-$\beta$ expansion
\[X=\sum_{n=1}^\infty X_n\beta^{-n}\]
where its $n$-th digit $X_n$ is defined as follows. For $x\in[0,1)$, let $T_\beta(x)=x\beta-\lfloor x\beta\rfloor$ be the base-$\beta$ transformation where $\lfloor\cdot\rfloor$ is the floor function, and let $T_\beta^n=T_\beta\circ\cdots\circ T_\beta$ denote the composition of $T_\beta$ with itself $n$ times. Then
 $X_n=\lfloor \beta T_\beta^{n-1}(X)\rfloor$ where $T_\beta^{0}(X)=X$. Furthermore, let
  $P_n$ denote the law of $T_\beta^{n}(X)=\beta^{n}\sum_{k=n+1}^\infty X_k\beta^{-k}$, where we interpret $\sum_{k=n+1}^\infty X_k\beta^{-k}$ as `the $n$-th remainder if we keep the $n$ first digits'. In \cite{paper1} we studied the convergence properties of $P_n$ when $\beta$ is an integer $\beta =q\ge2$:
 as $n\rightarrow\infty$, we proved that under mild conditions $P_n$ converges in total variation distance  to the uniform measure on $[0,1]$ and under certain smoothness conditions it converges with rate $q^{-n}$. It is natural to ask if such results extend to the case where $\beta>1$ is not an integer.
 Base-$\beta$
expansions where $\beta>1$ is not an integer are of interest 
  in symbolic dynamics, see \cite{Blanchard1989,Parry1966}.

In the present paper we consider the special case where $\beta=(1+\sqrt 5)/2$ is the golden ratio.
We show that when the distribution of $X$ is absolutely continuous with respect to Lebesgue measure, $P_n$  converges in the total variation distance to the absolutely continuous probability measure $P_\beta$ on $[0,1]$ with probability density function (pdf) 
\begin{equation}\label{e:f}
	f_\beta(x)=\begin{cases}
		(1+\beta)/\sqrt{5} & \text{if }0\le x<\beta^{-1},\\ \beta/\sqrt{5} & \text{if }\beta^{-1}\le x<1.
	\end{cases}
\end{equation}
It is well-known that $P_\beta$ is the unique absolutely continuous distribution which is
invariant under $T_\beta$, since $T_\beta$ is ergodic with respect to Lebesgue measure, see e.g.\ \cite{Karma}, where an equivalent expression for $f_\beta$ is given. Hence this result is analogous to the ones in \cite{paper1}.
Furthermore, we verify that when the pdf of $X$ is differentiable with bounded derivative, the convergence rate is exponential, but with a less obvious exponent than in \cite{paper1}. Also the proofs become more complicated when $\beta$ is the golden ratio because the way that base-$\beta$ expansions correspond to nested partitions of $[0,1)$ is more complex, as detailed in Section \ref{s:lemmas}. The proof of the main result (Theorem~\ref{t:3} in Section~\ref{s:results-beta}) exploits the nice properties of the golden ratio including its relation with the Fibonacci numbers. 

We leave the following open
problems for future research. For cases where $\beta>1$ is neither an integer nor the golden ratio  it may be interesting to study the convergence properties of $P_n$, but partitions of $[0,1)$ corresponding to the base-$\beta$ expansion will be even more complicated than if $\beta$ is the golden ratio or an integer greater than or equal to 2, see \cite{Parry1960}. Another, possibly more important, case
of interest is  obtained by defining
 $T(x)=1/x-\lfloor 1/x\rfloor$ and considering  the $n$-th remainder $T^{n}(X)=1/(X_n+1/(X_{n+1}+....))$ 
 obtained from
the continued fraction expansion  
 $X=1/(X_1+1/(X_2+...))$ where
  $X_n=\lfloor T^{n-1}(X)\beta\rfloor$.

\section{Some background on $\beta$-expansions}\label{s:lemmas}
For the remainder of this paper, let $\beta=(1+\sqrt5)/2$. For any number $x\in [0,1)$, $T_\beta$ defines  a $\beta$-expansion $x=\sum_{n=1}^\infty j_n \beta^{-n}$ where the digits $j_n \in \{0,1\}$ satisfy $j_n = \lfloor \beta T_\beta^{n-1}(x)  \rfloor$, $n=1,2,\ldots$ In fact there is a one-to-one correspondence between $\beta$-expansions of of this form and the set of digit sequences 
$$\Omega=\{(j_1,j_2,...)\,|\,j_n\in\{0,1\}\mbox{ and }j_nj_{n+1}=0\mbox{ for }n=1,2,...\},$$ 
see \cite[Section 2.2]{Fritz}. In this section, we introduce some more notation and prove two lemmas that will be needed for the main result in Section \ref{s:results-beta}.  

For $n=1,2,...$, let 
$$\Omega_n=\{(j_1,...,j_n)\in\{0,1\}^n\,|\, j_k j_{k+1}=0\mbox{ for }1\le k<n\}$$ 
which corresponds to the set of base-$\beta$ fractions $\sum_{k=1}^n j_k\beta^{-k}$ of order at most $n$ and where $(j_1,...,j_n,0,0,...)\in\Omega$. 
For each $J=(j_1,...,j_n)\in\Omega_{n}$, define $L_{n,J}=\sum_{k=1}^n j_k\beta^{-k}$. Denote by $\prec$ the lexicographic order on $\Omega_n$ and let $\hat J=(\hat j_1,...,\hat j_n)$ be the maximal element of $\Omega_n$, that is, $\hat j_k=1$ if $k$ is odd and $\hat j_k=0$ if $k$ is even.
Finally, define the intervals 
$$I_{n,J}=I_{j_1,...,j_n}=[L_{n,J},L_{n,J}+ j_n\beta^{-n-1}+(1- j_n)\beta^{-n}),\quad J=(j_1,...,j_n)\in\Omega_n.$$ 

The following lemma shows that the usual ordering of the numbers $L_{n,J}$ with $J\in\Omega_n$ agrees with that induced by $\prec$ via the bijective mapping $J\mapsto L_{n,J}$. Moreover, the intervals $I_{n,J}$ form a disjoint partition of $[0,1)$.

\begin{lemma}\label{l:2a} We have
\begin{equation}\label{e:lnhatJ}
L_{n,\hat J}+\hat j_n\beta^{-n-1}+(1-\hat j_n)\beta^{-n}=1.
\end{equation}
Furthermore, for every $J=(j_1,...,j_n)\in\Omega_n$ with $J\prec\hat J$, if $J'\in\Omega_n$ is the smallest element with $J\prec J'$ then
\begin{equation}\label{e:J'}
L_{n,J'}=L_{n,J}+j_n\beta^{-n-1}+(1-j_n)\beta^{-n}<1.
\end{equation}
\end{lemma}

\begin{proof}
Since $\beta^2=\beta+1$, or equivalently $\beta(1-\beta^{-2})=1$, $L_{n,\hat J}$ is equal to
\[\begin{cases}
\sum_{i=1}^{n/2}\beta^{-2i+1}=[{1-\beta^{-n}}]/[{\beta(1-\beta^{-2})}]=1-\beta^{-n} & \text{if }n\text{ is even},\\
\sum_{i=1}^{(n+1)/2}\beta^{-2i+1}=[{1-\beta^{-n-1}}]/[{\beta(1-\beta^{-2})}]=1-\beta^{-n-1} & \text{if }n\text{ is odd}.
\end{cases}
\]
 Thereby \eqref{e:lnhatJ} follows. 
 
 Let  $J=(j_1,...,j_n)\in\Omega_n$ with $J\prec\hat J$, and define $J'=(j_1',...,j_n')$ as in Lemma~ \ref{l:2a}. Let $j_k$ be the first digit where $J$ and $J'$ differ and $J_k=(j_1,\ldots,j_k)$. Then we must have $j_{k-1}=j_k=0$ and $j_k'=1$. Since $J'$ is minimal with $J\prec J'$, we must have that $j_{k-1}$ and $j_{k}$ is the last time two zeros appear in a row in $J$ and that all digits following $j_k'$ in $J'$ are zero. Hence, if $n-k$ is even, then $j_k=0$ implies $j_n=0$ and
 \begin{align*}
 	L_{n,J} {}&= L_{k,J_k} + \sum_{i=0}^{(n-k-2)/2}\beta^{-(k+1)-2i}=L_{k,J_k} + \beta^{-k}-\beta^{-n} \\
 	L_{n,J'} {}&= L_{k,J_k} + \beta^{-k}.
 \end{align*}
 Thus, \eqref{e:J'} holds when $n-k$ is even. 
If $n-k$ is odd, then $j_n=1$ and  
  \begin{align*}
 	L_{n,J} {}&= L_{k,J_k} + \sum_{i=0}^{(n-k-1)/2}\beta^{-(k+1)-2i}=L_{k,J_k} + \beta^{-k}-\beta^{-n-1} \\
 	L_{n,J'} {}&= L_{k,J_k} + \beta^{-k}.
 \end{align*}
This shows \eqref{e:J'} when $n-k$ is odd.
\end{proof}

For $j=0,1$, define the set $\Omega_{n,j}=\{(j_1,...,j_n)\in\Omega_n\,|\,j_n=j\}$ and its cardinality $N_j(n)=|\Omega_{n,j}|$, 
so $N(n)=|\Omega_n|$ is given by $N(n)=N_0(n)+N_1(n)$. Lemma~\ref{l:3} below determines these cardinalities 
in terms of the Fibonacci
sequence given by $b_0=0$, $b_1=1$, and $b_n=b_{n-1}+b_{n-2}$ for $n=2,3,...$ Recall that 
\begin{equation}\label{e:bn}
	b_n=[{\beta^n-(-\beta)^{-n}}]/{\sqrt 5},\quad n=0,1,2,...
\end{equation}

\begin{lemma}\label{l:3} For $n=1,2,...$, we have
\begin{equation}\label{e:fib}
N_1(n)=b_n,\quad N_0(n)=b_{n+1}, 
\quad N(n)=b_{n+2}.
\end{equation} 
\end{lemma} 
\begin{proof} We have $N_0(1)=N_1(1)=N_1(2)=1$ and $N_0(2)=2$. Let $(j_1,...,j_n)\in\Omega_n$ with $n>1$.
If $j_{n-1}=1$ then $j_n=0$ and $I_{j_1,...,j_{n-1}}=I_{j_1,...,j_n}$ is unchanged, whilst if $j_{n-1}=0$ then $I_{j_1,...,j_{n-1}}$ splits into $I_{j_1,...,j_{n-1},0}$ and $I_{j_1,...,j_{n-1},1}$. Therefore, $N_1(n)=N_0(n-1)$ and $N_0(n)=N_1(n-1)+N_0(n-1)$ whenever $n>1$.  
Consequently, $N_1(n)=N_1(n-1)+N_1(n-2)$ satisfies the recursion of the Fibonacci numbers for $n>1$, and since $N_1(1)=N_1(2)=1$, the first identity in \eqref{e:fib} follows. This implies the second and third identities in \eqref{e:fib}, since $N_0(n)=N_1(n+1)$ and $N(n)=N_0(n)+N_1(n)= b_n + b_{n+1}=b_{n+2}$ for $n\ge 1$.    
\end{proof}

\section{Distribution of the remainder}
\label{s:results-beta}

Let $X=\sum_{n=1}^\infty X_n\beta^{-n}$ be a random variable on $[0,1)$ and let $F$ denote the cumulative distribution function (CDF) of $X$. Let $F_n$ be the CDF of $T_\beta^n(X)$.

\begin{lemma}\label{l:2} Assume that for every $J\in\Omega_n$, $F$ has no jump at the endpoints of the interval $I_{n,J}$. Then for every $x\in[0,\beta^{-1})$ we have
	\begin{equation}\label{e:Fn-beta1}
		F_n(x)=\sum_{J\in\Omega_n}F\left(L_{n,J}+x\beta^{-n}\right)-F\left(L_{n,J}\right),
	\end{equation}
	and for every $x\in[\beta^{-1},1)$ we have
	\begin{equation}\label{e:Fn-beta2}
		F_n(x)=\sum_{J=(j_1,...,j_n)\in\Omega_n}F\left(L_{n,J}+(1-j_n)x\beta^{-n}+j_n\beta^{-n-1}\right)-F\left(L_{n,J}\right).
	\end{equation}
\end{lemma}
\begin{proof}
	By Lemma~\ref{l:2a}, the collection of intervals 
	$I_{n,J}$ with $J\in\Omega_n$ 
	provides a subdivision of $[0,1)$, where $I_{n,J}$ has length $\beta^{-n}$ if $j_n=0$, and $\beta^{-n-1}$ if $j_n=1$. Let $x\in[0,1)$ and $J\in\Omega_n$.
	If $x<\beta^{-1}$ then $x\beta^{-n}\le\beta^{-n-1}$, and if $x\ge\beta^{-1}$ then $x\beta^{-n}\ge\beta^{-n-1}$. So if $x<\beta^{-1}$ or if both $x\ge\beta^{-1}$ and $J\in\Omega_{n,0}$, then 
	$X\in I_{n,J}$ and $T^n(X)\le x$ if and only if $L_{n,J}\le X\le L_{n,J}+x\beta^{-n}$. However, if $x\ge\beta^{-1}$ and $J\in\Omega_{n,1}$, then 
	$X\in I_{n,J}$ and $T^n(X)\le x$ if and only if $L_{n,J}\le X\leq L_{n,J}+x\beta^{-n-1}$. Now, using the assumption in the lemma, we immediately obtain \eqref{e:Fn-beta1} and \eqref{e:Fn-beta2}.
\end{proof}

\subsection{Main results}\label{s:main}

Henceforth, assume $X$ has a pdf $f$ on $[0,1)$. By Lemma~\ref{l:2}, $T_\beta^n(X)$ is absolutely continuous with pdf given by
\begin{equation}\label{e:newfn}
f_n(x)=\begin{cases}\sum_{J\in\Omega_n}\beta^{-n}f\left(L_{n,J}+x\beta^{-n}\right)&\text{if }0\le x<\beta^{-1},\\ 
\sum_{J\in\Omega_{n,0}}\beta^{-n}f\left(L_{n,J}+x\beta^{-n}\right)&\text{if }\beta^{-1}\le x<1.
\end{cases}
\end{equation}
Recall that $P_n$ denotes the probability measure defined by $f_n$ and that $f_{\beta}$ denotes the pdf given by \eqref{e:f}.

\bp 
Assume that for some integer $m\ge1$, $f$ is Lebesgue almost everywhere 
equal to some constant $C_J\geq 0$ on each interval $I_{m,J}$ with $J\in\Omega_m$. In addition assume that 
\begin{equation}\label{eq:con}
 \sum_{J\in \Omega_{m,0}} C_J = \beta^{m}\frac{1+\beta}{\sqrt{5}}, \quad   \sum_{J\in \Omega_m} C_J = \beta^m \frac{\beta}{\sqrt{5}}.
 \end{equation}
Then $f_n=f_\beta$ for $n=m,m+1,...$.
\ep

\begin{proof}
	First, let $y\in [0,\beta^{-1})$. Then 
	\begin{equation*}
	f_m(y) = \sum_{J\in\Omega_m}\beta^{-m}f\left(L_{m,J}+x\beta^{-m}\right) = \beta^{-m}\sum_{J\in\Omega_m}C_J =f_\beta(y),
	\end{equation*}
	where \eqref{e:newfn} was used in the first equality, the second equality  used the first assumption of the proposition together with the fact that $L_{m,J}+x\beta^{-m}\in I_{m,J}$, and the third equality used the assumption \eqref{eq:con}.
	
	Second, 
	for $y\in [\beta^{-1},1)$, by similar arguments,
	\begin{equation*}
	f_m(y) = \sum_{J\in\Omega_{m,0}} \beta^{-m} f\left(L_{m,J}+x\beta^{-m}\right) = \beta^{-m}\sum_{J\in\Omega_{m,0}}C_J =f_\beta(y).
	\end{equation*}
	Thus, the proposition follows for $n=m$. For  $n>m$, we just use that $f_\beta$ is $T_\beta$-invariant.
\end{proof}

We need some notation for the following theorem. For a real function $g$ defined on $[0,1]$, denote its $L_1$- and supremum-norm
by
 $\|g\|_1=\int_0^1|g(t)|\,\mathrm dt$ and
 $\|g\|_\infty=\sup_{x\in[0,1]}|g(x)|$, respectively, and denote the corresponding $L_1$-space by $L_1([0,1])=\{g\,|\,\|g\|_1<\infty\}$. Let $D([0,1])=\{g\geq 0\,|\,\|g\|_1=1\}$ be the subspace of pdfs, and  $CD([0,1])\subset D([0,1])$ its subspace of functions $g$ which are continuous on $[0,1]$ and differentiable on the open interval $(0,1)$ with $\|g'\|_\infty < \infty$. For $g\in CD([0,1])$, define  
 $g_n$ as in \eqref{e:newfn} with $f$ replaced by $g$.

	Let $\mathcal{B}$ denote the Borel subsets of $[0,1)$.
	 Define the total variation distance 
	 $$d_{\mathrm{TV}}(P_n,P_\beta) = \sup_{A\in \mathcal{B}}|P_n(A) - P_\beta(A)| = \frac{1}{2}\|f_n-f_\beta\|,$$ 
	see e.g.\ Lemma~2.1 in \cite{Tysbakov} for the second equality. 
	
	\begin{theorem}
			\label{t:3} 
			If $g \in CD([0,1])$ then
		\begin{equation}\label{e:r1}
		\|f_n-f_\beta\|_1 \le \|f-g\|_1 + O\left(\beta
		^{-2n/3 } (\|g'\|_\infty + 1) \right).
		\end{equation}
		In particular, 
		\begin{equation}\label{e:r2}
		\lim _{n\to \infty} d_{\mathrm{TV}}(P_n,P_\beta) = 0.
		\end{equation}
		Furthermore,  if $f \in CD([0,1])$ then $P_n$ converges exponentially fast:
		\begin{equation}\label{e:r3}
		d_{\mathrm{TV}}(P_n,P_\beta) = O\left( \beta^{-2n/3}( \|f'\|_\infty +1) \right).
		\end{equation}

	\end{theorem}
	
	\begin{proof}
		Let $g\in CD([0,1])$. We have 		
		\begin{align}\nonumber
		&\int_0^{\beta^{-1}} |f_n(x) - g_n(x)|\, \mathrm d x \\ \nonumber
		&\le \sum_{J\in \Omega_n}\beta^{-n}\int_0^{\beta^{-1}} |f(L_{n,J} + x\beta^{-n}) - g(L_{n,J} + x\beta^{-n})|\, \mathrm d x\\
		 &= \sum_{J\in \Omega_{n,0}}\int_{L_{n,J}}^{L_{n,J}+\beta^{-n-1}} |f( u) - g(u)|\, \mathrm d u + \sum_{J\in \Omega_{n,1}} \int_{I_{n,J}} |f(u) - g (u)|\,\mathrm d u, \label{e:fn-gn}
		\end{align}
		while 
		\begin{equation}\label{e:fn-gn2}
		\int_{\beta^{-1}}^1 |f_n(x) - g_n(x)|\, \mathrm d x  \le \sum_{J\in \Omega_{n,0}} \int_{L_{n,J}+\beta^{-n-1}}^{L_{n,J}+\beta^{-n}} | f(u) - g(u)|\, \mathrm d u.
		\end{equation}
		Adding \eqref{e:fn-gn} and \eqref{e:fn-gn2} and using Lemma \ref{l:2a} gives
		\begin{equation}\label{fn-gn}
		\|f_n - g_n\|_1 \le \|f-g\|_1.
		\end{equation}
			
		 Set $n = n_1 + n_2$. For $J\in \Omega_{n_1}$ define $\Omega_n(J)$ to be the set of $(j_1,\ldots,j_n)\in \Omega_n$ for which $(j_1,\ldots,j_{n_1})=J$.  Note that $I_{n_1,J}= \bigcup_{J' \in \Omega_{n}(J)} I_{n,J'}$.  Similarly, let $\Omega_{n,0}(J)$ and $\Omega_{n,1}(J)$ denote the sets consisting of $(j_1,\ldots,j_n)\in \Omega_n(J)$ with $j_n=0$ and $j_n=1$, respectively.  Then, for $x\in [0,\beta^{-1})$, \eqref{e:newfn} yields 
		\begin{align*}
		&g_n(x) %= \sum_{J\in \Omega_{n_1}} \left(\sum_{J'\in\Omega_{n,0}(J)} \int_{I_{J'}}g(t)\,\mathrm dt 
		%+ \sum_{J'\in\Omega_{n,1}(J)}\beta   \int_{I_{J'}}g(t)\,\mathrm dt \right) + O(\beta^{-n+2})||g'||_\infty\\
		= \sum_{J\in \Omega_{n_1}}\sum_{J'\in \Omega_{n}(J)} \beta^{-n} g(L_{n_1,J'}+ \beta^{-n}x)  \\
		 %&+\sum_{J'\in \Omega_{n,1}(J)} \left(\beta^{-n} g(L_{n_1,J}) + O(\beta^{-(n_1+ n+1)})\|g'\|_\infty \right)\bigg)\\ %+ O(\beta^{-n+2})\|g'\|_\infty\\
		 &= \bigg(\sum_{J\in \Omega_{n_1}}\sum_{J'\in \Omega_{n}(J)} \beta^{-n}|I_{n_1,J}|^{-1}\int_{I_{n_1,J}}g(t) \, \mathrm d t \bigg)+ O(\beta^{-n_1+2})\|g'\|_\infty, 
		 \end{align*}
		 where in the second equality we used that by the mean value theorem
		 $$ |g(L_{n,J'} + \beta^{-n}x) - g(t)| \le \|g'\|_\infty|I_{n_1,J}| \le \|g'\|_\infty\beta^{-n_1} ,  $$ 
		 for $J\in \Omega_{n_1}$, $J' \in \Omega_n(J)$ and $t\in I_{n_1,J}$, and the fact that by Lemma \ref{l:3}, $N(k)\in O(\beta^{k+2})$. Furthermore,
		\begin{align} \nonumber
		&g_n(x)= \bigg(\sum_{J\in \Omega_{n_1,0}}N(n_2)\beta^{-n + n_1} \int_{I_{n_1,J}}g(t) \mathrm d t \\ \nonumber
		&+ \sum_{J\in \Omega_{n_1,1}}N(n_2-1)\beta^{-n+n_1+1}  \int_{I_{n_1,J}}g(t) \mathrm d t\bigg) + O(\beta^{-n_1+2})\|g'\|_\infty\\  \label{gn1}
          & = \sum_{J\in \Omega_{n_1}}\bigg(\frac{\beta + 1}{\sqrt{5}}\int_{I_{n_1,J}}g(t) \mathrm d t\bigg) +   O(\beta^{-n_1+2}\|g'\|_\infty) + O(\beta^{-2n_2}),
		\end{align}
		where we used that $\beta^{-k} N(k) = \beta^2/\sqrt{5} + O(\beta^{-2k-2}) = (1+\beta)/\sqrt{5} + O(\beta^{-2n_2-2}) $ by Lemma~\ref{l:3} and that $\sum_{J \in \Omega_{n_1}} \int_{I_{n_1,J}}|g(t)| \mathrm d t = 1$.
		
		For  $x\in [\beta^{-1},1)$, a similar computation yields
		\begin{align}\nonumber
		&g_n(x)
	     =\bigg(\sum_{J\in \Omega_{n_1}} \sum_{J'\in \Omega_{n,0}(J)} \beta^{-n}|I_{n_1,J}|^{-1} \int_{I_{n_1,J}}g(t) \,\mathrm d t  \bigg) + O(\beta^{-n_1+1}) \|g'\|_\infty\\ \nonumber
		%&=\sum_{J\in \Omega_{n_1}}\sum_{J'\in \Omega_{n,0}(J)} \beta^{-n} |I_{J}|^{-1}\int_{I_{J}}g(t) \,\mathrm d t + O(\beta^{-n_1+1}) ||g'||_\infty.\\ \nonumber
		&=\bigg(\sum_{J\in \Omega_{n_1,0}} \beta^{-n + n_1}  N_0(n_2) \int_{I_{n_1,J}}g(t) \,\mathrm d t  \\ \nonumber
		&+\sum_{J\in \Omega_{n_1,1}} \beta^{-n + n_1+1} N_0(n_2-1) \int_{I_{n_1,J}}g(t)\,\mathrm  d t \bigg) +  O(\beta^{-n_1+1}) \|g'\|_\infty\\
		&=\frac{\beta}{\sqrt{5}} + O(\beta^{-2n_2+1})  + O(\beta^{-n_1+1}) \|g'\|_\infty, \label{gn2}
		\end{align}
		where we used that $\beta^{-k} N_0(k) = \beta/\sqrt{5} + O(\beta^{-2k-1})  $ by Lemma~\ref{l:3}. 
		
		Combining \eqref{gn1} and \eqref{gn2} shows 
		\begin{equation}\label{gn-fT}
		g_n(x) = f_\beta(x) + O(\beta^{-n_1+2})\|g'\|_\infty  + O(\beta^{-2n_2+1}).
		\end{equation}
				 Set $n_2 = \lfloor n/3 \rfloor$.  Using \eqref{fn-gn} and \eqref{gn-fT}, we get 
		\begin{align*}
		\|f_n-f_\beta\|_1 \le \|f_n-g_n\|_1 + O(\beta^{-2n/3+3})(\|g'\|_\infty+ 1) \\
		\le \|f-g\|_1 + O(\beta^{-2n/3+3})(\|g'\|_\infty + 1), 
		\end{align*}
		 which shows \eqref{e:r1}  and \eqref{e:r3}. We let $n\to \infty$  and obtain 
		\begin{equation*}
		\limsup_{n\to \infty} \|f_n-f_\beta\|_1 \le \|f-g\|_1.
		\end{equation*}
		The right hand side can be made arbitrarily small, which shows \eqref{e:r2}.   
\end{proof}

\section*{Acknowledgements}
Supported by The Danish Council for Independent Research —
Natural Sciences, grant DFF – 10.46540/2032-00005B.

\end{document}